\documentclass[12pt]{amsart}
\usepackage{amssymb,amsmath,amsfonts}
\usepackage{mathrsfs}
\usepackage{ae,aecompl}
\usepackage[T1]{fontenc}
\usepackage[pdftex]{graphicx}
\usepackage{enumerate}

\renewcommand{\mathcal}{\mathscr}
\renewcommand{\epsilon}{\varepsilon}

\newcommand{\dimH}{\dim_H}

\newcommand{\EE}{\mathbb E}
\newcommand{\PP}{\mathbb P}
\newcommand{\N}{\mathbb N}
\newcommand{\ag}{\Gamma}

\DeclareMathOperator{\dist}{dist}

\newtheorem{theorem}{Theorem}[section]
\newtheorem{lemma}[theorem]{Lemma}

\newtheorem{proposition}[theorem]{Proposition}

\theoremstyle{definition}

\newtheorem{notation}[theorem]{Notation}

\theoremstyle{remark}
\newtheorem{remark}[theorem]{Remark}
\newtheorem*{localdim}{Proof of Lemma \ref{lem:localdim}}

\numberwithin{equation}{section}

\begin{document}

\title{Projections of random covering sets}
\author{Changhao Chen$^1$, Henna Koivusalo$^2$, Bing Li$^{3,1}$, Ville Suomala$^1$}
\address{$^1$Department of Mathematical Sciences, P.O. Box 3000, 90014
University of Oulu, Finland}
\address{$^2$Department of Mathematics,
University of York,
Heslington, York, YO10 5DD,
UK}
\address{$^3$Department of Mathematics, South China University of Technology,
Guangzhou, 510641, P.R. China}

\email{changhao.chen@oulu.fi, henna.koivusalo@york.ac.uk}
\email{scbingli@scut.edu.cn, ville.suomala@oulu.fi}

\date{\today}

\begin{abstract}
We show that, almost surely, the Hausdorff dimension $s_0$ of a random covering set is preserved under all orthogonal projections to linear subspaces with dimension $k>s_0$. The result holds for random covering sets with a generating sequence of ball-like sets, and is obtained by investigating orthogonal projections of a class of random Cantor sets.
\end{abstract}
\maketitle
\noindent{\small{\bf AMS Subject Classification (2010)}:\ 60D05,
28A78, 28A80.}

\section{Introduction}

We begin by giving a definition for random covering sets. Given a sequence of independent random variables $(\xi_n)$, uniformly distributed on the torus $\mathbb T^d$, and a sequence of subsets of the torus, $(g_n)$, random covering set $E$ is the set of infinitely often covered points,
\[
E=\limsup _{n\to\infty}(\xi_n + g_n)=\bigcap_{k=1}^\infty \bigcup_{n=k}^\infty(\xi_n + g_n).
\]
Here we interpret $\xi_n + g_n\subset \mathbb T^d$. For further
details and background on random covering sets, see e.g. the survey \cite{Kahane00}. Here we only mention a few key observations.

It is an immediate consequence of Borel-Cantelli lemma and Fubini's theorem that
\[
\mathcal{L}(E)=
\begin{cases}
0,\textrm{ when }\sum_{n=1}^\infty \mathcal{L}(g_n)<\infty\\
1, \textrm{ when }\sum_{n=1}^\infty \mathcal{L}(g_n)=\infty
\end{cases},
\]
where $\mathcal{L}$ is the Lebesgue measure.
In the earlier research on the random covering sets, the main emphasis has been on the full-measure case in the circle $\mathbb T^1$. Of particular interest have been variants of a problem posed by Dvoretzky in 1956 asking what kind of conditions on $g_n$ guarantee that
\[
\mathbb P(E=\mathbb T^1)=1\,.
\]
This particular problem was fully solved by Shepp in 1972 \cite{Shepp72}, but many related problems especially in higher dimensions are still open.

Recently, also the zero measure case has drawn a lot of attention among the researchers. The Hausdorff dimension of the random covering sets in $\mathbb T^1$ when $g_n$ are intervals of length $1/n^\alpha$ was first calculated by Fan and Wu \cite{FanWu04}. Durand \cite{Durand10} studied the Hausdorff measure and large intersection properties and reproved the dimension result in $\mathbb T^1$. For certain box-like random covering sets in $\mathbb{T}^d$, the dimension was obtained by of J\"arvenp\"a\"a, J\"arvenp\"a\"a, Koivusalo, Li and Suomala \cite{JJKLS}.
Recently, Persson \cite{Persson} generalised the result using the large intersection method.

In this article we study the dimensions of orthogonal projections of the random covering sets (after embedding the torus $\mathbb{T}^d$ into $\mathbb{R}^d$ in a natural way). Our motivation stems from the classical results of projections of general sets by Marstrand, Kauffman, Mattila and others. In particular, if the Hausdorff dimension of a Borel set $A\subset\mathbb{R}^d$, $\dimH A=s$, then it is well known that for almost all $k$-dimensional subspaces $V\subset\mathbb{R}^d$ the orthogonal projection of $A$ into $V$, $\pi_V A$ is of Hausdorff dimension $\min\{s,k\}$. Below, we refer to this fact as the \emph{projection theorem}. For background and references, see e.g. \cite{Mattila95}. One should bear in mind that since this is an almost all type statement, for most concrete sets there are plenty of exceptional directions for which the dimension is less than the expected value. Moreover, if one fixes a direction, the projection theorem cannot be used to obtain information about the dimension of the projection in this direction.

However, for suitable random families of sets, the situation is different. A model example is the fractal percolation for which Falconer \cite{Falconer89} and Falconer and Grimmett \cite{FalconerGrimmett92} showed that for the principal directions the dimension of the projection of the limit set is almost surely equal to the expected value.
It is harder to show, and this was done only recently by Rams and Simon (for $d=2$) \cite{RamsSimonb}, \cite{RamsSimona}, that the same remains true for all directions simultaneously. That is, if $s$ is the almost sure dimension of the fractal percolation limit set $F\subset\mathbb{R}^2$, then almost surely, $\dimH \pi_V (F)=\min\{s,1\}$ for all lines $\ell\subset\mathbb{R}^2$.

Our approach is motivated by Shmerkin and Suomala \cite{ShmerkinSuomalaa} where a variant of the fractal percolation model is used in order to study the dimension of non-tube null sets. Some ideas have also been adapted from the forthcoming work \cite{ShmerkinSuomalab}, where Shmerkin and Suomala prove various results for the intersections and projections of families of random martingale measures (for which fractal percolation is a central example).

In our main result, Theorem \ref{randomcoveringproj}, we establish that for the random covering sets, there are almost surely no exceptional directions for the projection theorem. This, on the other hand, is a corollary of a similar result for a general class of random Cantor sets, Theorem \ref{main}. The main difference between our model and the (variations of) fractal percolation in \cite{RamsSimona}, \cite{RamsSimonb}, \cite{ShmerkinSuomalaa}, \cite{ShmerkinSuomalab} is the fact that we do not impose any bounds on how fast the diameter of the basic sets (or construction cubes) goes to zero. This is essential for the application to the random covering sets and also the main reason why we cannot derive the result directly from \cite{ShmerkinSuomalab}.

The paper is organised as follows. The random Cantor sets are introduced in Section \ref{sect:preliminaries} together with the required notation and definitions. The main result for these Cantor sets is also presented as Theorem \ref{main}. In Section \ref{sec:geom} we present some geometric lemmata, which have been adapted from \cite{ShmerkinSuomalaa} to the present setting. Section \ref{sect:probability} contains the main probabilistic argument, and concludes the proof of Theorem \ref{main}. The dimension result for the projections of the random covering sets is then derived as a corollary in Section \ref{sect:application}, see Theorem \ref{randomcoveringproj}. Throughout Sections \ref{sect:preliminaries}--\ref{sect:application} we only consider the case $d=2$. Higher dimensional generalisations are discussed at the end in Section \ref{sect:general}.

\section{Random Cantor sets and their projections}\label{sect:preliminaries}
In this section, we define the random Cantor type sets that we consider and state the main results for the dimension of their orthogonal projections.

Given $k\in\N$, we call the closed squares $[\tfrac{i}{k},\tfrac{i+1}{k}]\times[\tfrac{j}{k},\tfrac{j+1}{k}]$, $0\le i,j\le k-1$, \emph{$k$-adic}. In a similar way, we can consider the \emph{$k$-adic grid} of any $Q=[x,x+\lambda]\times[y,y+\lambda]\subset[0,1]^2$ consisting of the $k^2$ closed subsquares of $Q$ of side-length $\lambda/k$.

Given $M\in \mathbb N$ and $0<s<1$, we consider the following random model. Decompose the unit square into $M$-adic subsquares and randomly choose $N\leq M^2$ of these subsquares $Q_1,\ldots, Q_N$, all choices being independent and uniformly distributed (i.e. $\mathbb{P}(Q_i=Q)=M^{-2}$ for each $Q$ in the grid and each $1\le i\le N$). We stress that the chosen subsquares need not be disjoint, that is, the same square is allowed to be chosen multiple times.

Let $(M_n)\subset \mathbb N$ and $(N_n)\subset \mathbb N$ be sequences of integers with $N_n\leq M_n^2$ ($M_n\ge 2$) for all $n$ and
\begin{equation}\label{eq:def of s}
s=\liminf _{n\to\infty} \frac{\sum_{i=1}^n\log N_i}{\sum_{i=1}^n\log M_i}<1.
\end{equation}
Let
\begin{align}\label{r_n}
r_n&=\left(\prod_{k=1}^n M_k\right)^{-1}\,,\\
\label{P_n}
 P_n&=\prod_{k=1}^n N_k\,.
\end{align}

We consider a random construction obtained by iterating the above construction with $M=M_n$ and $N=N_n\le M_n^{2}$. We first choose randomly $\mathcal{F}_1=(Q_{1,1},\ldots,Q_{1,N_1})$ among the squares in the $M_1$-adic grid of the unit square as in above. Suppose that for $n\in\mathbb N$, $\mathcal{F}_n=(Q_{n,1},\ldots,Q_{n,P_n})$ are chosen $r_{n}^{-1}$-adic squares.    Independently inside each of these $Q_{n,j}$, we consider the $M_{n+1}$-adic grid and perform the above construction with $M=M_{n+1}$ to obtain $N_{n+1}$ chosen subsquares of each $Q_{n,j}$. Let $\mathcal{F}_{n+1}$ consist of all the chosen subsquares inside the $Q_{n,j}\in\mathcal{F}_n$. Again, it should be pointed out that the elements of $\mathcal{F}_{n+1}$ do not have to be disjoint. For this reason we use the notation $(Q_{n,1},\ldots)$ rather than $\{Q_{n,1},\ldots\}$. For convenience, we however stick to notations like $Q\in\mathcal{F}_n$ instead of the more rigorous $Q=Q_{n,i}$, $i\in\{1,\ldots, P_n\}$.

Denote by $\omega$ the elements in the probability space $\Omega$ induced by the construction described above. Let $F=F(\omega)$ be the random limit set
\[
F=\bigcap_{n=1}^\infty \bigcup_{Q\in\mathcal F_n} Q.
\]
For each line $\ell$ in the plane, denote by $\pi_\ell$ the orthogonal projection to $\ell$ and by $\dimH (F)$ the Hausdorff dimension of a set $F$. Our main result for the random Cantor sets is the following.

\begin{theorem}\label{main}
Almost surely for all lines $\ell$, $\dimH \pi_\ell(F)= s$.
\end{theorem}

\begin{remark}
It is an immediate consequence of the definitions that we have the dimension upper bound $\dimH(F)\le s$ for all $F$. It is more or less standard to show that almost surely, also $\dimH(F)\ge s$ (and these estimates remain true also for $1\le s\le 2$). We do not give the proof here, since in the case $s<1$ that is relevant for us, Theorem \ref{main} yields a much stronger result.

When the sequence $(M_n)$ is bounded, Theorem \ref{main} could be obtained by adapting the argument of \cite{ShmerkinSuomalaa}. It would also follow directly from the general results for random martingale measures in \cite{ShmerkinSuomalab}. Whence the main content in Theorem \ref{main} is that we do not impose any bounds on the growth of the numbers $M_n$. We also remark that if $M_n$ is bounded, then the packing dimension of $F$ almost surely equals $s'=\limsup_{n\to\infty} \frac{\sum_{i=1}^n\log N_i}{\sum_{i=1}^n\log M_i}$, whereas in the unbounded case the packing dimension can be anything in between $s'$ and $2$ depending on the growth speed of $M_n$.
\end{remark}

Let
\begin{equation}\label{c_n}
\mu_n=\sum_{Q\in\mathcal{F}_n}c_n\mathcal{L}|_{Q},\ \ \text{where}\ \
c_n=\prod_{i=1}^n M_{i}^{2}N_i^{-1}=r_n^{-2}P_n^{-1}
\end{equation}
 and $\mathcal{L}|_Q$ is the Lebesgue measure on $Q$. Then $\mu_n$ is a probability measure for each $n$ and since the measure of any $r_n^{-1}$-adic square remains unchanged after $n$ steps, it follows that the sequence $(\mu_n)$ converges in weak$^*$ topology to a random probability measure $\mu$ on $[0,1]^2$.

In the following, tubular neighbourhoods of lines in the unit cube are called strips. More precisely, a strip $S$ of width $w(S)=\delta>0$, defined by a line $\ell$, is the set
\[
S=\{x\in [0,1]^2\mid\dist(x,\ell)<\delta/2\}
\]
where $\dist$ is the Euclidean distance.

 If a line $\ell$ is given, we denote by $B(x,r)$ the ball of radius $r$ and center $x$ on the line $\ell$. For a line $\ell$ we use the notation
\[
|\ell\cap\mathcal F_n|:=\sum_{Q\in\mathcal F_n}\mathcal H^1(\ell\cap Q),
\]
where $\mathcal H^1$ is the $1$-dimensional Hausdorff measure. Since the same square can occur many times in $\mathcal{F}_n$, this is in general different than the length of $\ell\cap\bigcup_{Q\in\mathcal{F}_n}Q$. For a strip $S$, denote
\[
Z(S,n)=\# \{Q\in\mathcal F_n\mid Q\cap S\neq \emptyset\},
\]
where $\# J$ denotes the cardinality of a set $J$.  Denote the indicator function of an event $A$ by $\chi_A$ and the complement of $A$ by $A^c$. Finally, denote by $f_*\mu$ the image measure of $\mu$ under a mapping $f$.

Theorem \ref{main} is easily deduced from the following estimate for the projections of the limit measure $\mu$, which is proved in Sections \ref{sec:geom}--\ref{sect:probability}.

\begin{theorem}\label{lem:localdim}
Let $t<s$. There is a constant $C_0<+\infty$ such that there is a positive probability for the event
\[
(\pi_\ell)_*\mu(B(x,r))\leq C_0r^t\text{ for all lines }\ell,\text{ for all }x\in\ell\text{ and for all }r>0\,.
\]
\end{theorem}

\begin{proof}[Proof of Theorem \ref{main} assuming Theorem \ref{lem:localdim}]
Obviously $\dimH \pi_\ell(E)\le \dimH(E)\le s$ for all lines, so we concentrate on the lower bound.

For all $t<s$, by Lemma \ref{lem:localdim}  the estimate
\[
(\pi_\ell)_*\mu(B(x,r))\leq C_0r^t,
\]
holds with positive probability for all $\ell$, $x$ and $r$, and thus with positive probability a lower bound $\dimH \pi_\ell(F)\geq t$ holds for all lines $\ell$ (See e.g. \cite[4.2]{Falconer03}). Since the event ``$\dimH \pi_\ell(F)\geq t$ for all lines'' is a tail event, the Kolmogorov zero-one law implies that this lower bound holds almost surely. Approaching $s$ along a sequence gives, almost surely for all lines $\ell$, the lower bound $\dimH \pi_\ell(F)\geq s$.
\end{proof}

\section{Geometric Lemmata}\label{sec:geom}

In this section, we present some simple geometric observations. The following lemma is an adaptation of \cite[Lemma 3.3]{ShmerkinSuomalaa} to our setting. We do not repeat the proof here.

\begin{lemma}\label{lem:finitenumblines}
There is a collection of lines, $\mathcal A_n$, such that for any line $\ell$ there is $\tilde{\ell}\in\mathcal A_n$ with
\[
|\ell\cap \mathcal F_n|\leq |\tilde{\ell}\cap \mathcal F_n|+ r_n
\]
and $\mathcal A_n$ has at most $C r_{n}^{-4}$
elements, where $C<+\infty$ is a constant.
\end{lemma}

\begin{lemma}\label{lem:finite_approx_family}
Let $2\le M\in\mathbb{N}$.
There is a finite family of strips $\mathcal{D}$ with at most $16 M^{3}$ elements so that for any $M^{-1}\le\delta\le1$ and any strip $S$ of width $\delta$, there is $\widetilde{S}\in\mathcal{D}$ such that $w(\widetilde{S})\leq 5 w(S)$ and $S\subset\widetilde{S}$.
\end{lemma}
\begin{proof}
Denote by $\mathcal B$ the collection of all lines connecting any two disjoint points in the collection
\[
\begin{aligned}
\{(0, \tfrac kM)\in&\mathbb{R}^2\mid k, l=1,\dots, M\}\cup\{(1, \tfrac kM)\in\mathbb{R}^2\mid k, l=1,\dots, M\}\\
&\cup \{(\tfrac kM, 0)\in\mathbb{R}^2\mid k, l=1,\dots, M\}\cup \{(\tfrac kM, 1)\in\mathbb{R}^2\mid k, l=1,\dots, M\}.
\end{aligned}
\]
Denote by $\mathcal D^1$ the collection of all strips of width $5M^{-1}$ defined by the at most $16M^2$ lines in the collection $\mathcal B$. This has the desired property for all strips $S$ of width $M^{-1}\le w(S)\leq 2M^{-1}$.

Repeat this argument for all $i=2,\dots, M$ to obtain collections $\mathcal D^{i}$, that satisfy the claim for strips $S$ of width $iM^{-1}\le w(S) \le (i+1)M^{-1}$. Then choose $\mathcal D=\mathcal D^1\cup\dots \cup\mathcal D^M$.
\end{proof}

\begin{lemma}\label{lem:lengthtonum}
Let $0\le V\in\mathbb R$. If $|\ell\cap\mathcal F_n|\leq V$ for all lines $\ell$, then for a strip $S$ with width $0<w(S)\leq r_n$,
\[
Z(S,n) \leq 2(1+2\sqrt{2})r_n^{-1}V.
\]
\end{lemma}
\begin{proof}
Let $\ell$ be the line defining $S$, $\ell'$ a line perpendicular to $\ell$ with $\ell\cap \ell'=\{x\}$ for $x\in\mathbb R^2$ and $\ell_y$ lines parallel to $\ell$ with $\ell'\cap\ell_y=\{y\}$. Denote by $r$ the number $w(S)+2\sqrt 2r_n$ and by $B(x,r)$ the ball of center $x$ and radius $r$ on the line $\ell'$. Then Fubini's theorem, \eqref{c_n} and the assumption $|\ell'\cap\mathcal{F}_n|\le V$ imply
\begin{align*}
\mu_n(\pi_{\ell'}^{-1}(B(x,r)))&= c_n\int_{B(x,r)}|\ell_y\cap \mathcal F_n|\,dy\\
&\leq 2Vrc_n.
\end{align*}
If $Q\cap S\neq\emptyset$, then $Q\subset \pi_{\ell'}^{-1}(B(x,r))$, and from \eqref{c_n}, we have
\[
Z(S,n)P_n^{-1}\leq \mu_n(\pi_{\ell'}^{-1}(B(x,r)))\leq 2Vrc_n\,,
\]
as required.
\end{proof}

\section{Main proofs}\label{sect:probability}

Throughout this section, we fix a number $t<s$ and let $0<2\epsilon<s-t$. By slight abuse of notation, denote by $\EE(\cdot\mid \mathcal F_n)$ the expectation conditional on construction squares up to level $n$ being chosen, and by $\PP(\cdot \mid \mathcal F_n)$ the corresponding conditional probability.

\begin{notation}\label{events}
Let $\ell$ be a line, $S$ a strip and $n\in\mathbb N$. Let $G(\ell,n)$ be the event $|\ell\cap \mathcal F_n|\leq RP_{n} r_{n}^{t+1} + r_n$, where $R>0$ is a large constant. Denote by $A(S,n)$ the event $Z(S,n)\leq 500\cdot 5^t R P_{n}w(S)^t$. Let $G_n$ be the event that $G(\ell,n)$ occurs for all lines $\ell$, and $A_{n}$ the event that $A(S,n)$ occurs for all strips $S$ with $\tfrac 15 r_{n}<w(S)\le \tfrac 15r_{n-1}$. Finally, let $\ag_n$ be the event $A_n\cap G_n$.
\end{notation}

In order to prove Theorem \ref{lem:localdim} we wish to verify that
\begin{equation}\label{positiveprobability}
\PP(\bigcap_{n=1}^\infty\ag_n)>0,
\end{equation}
since then for all $\omega\in\bigcap_{n=1}^\infty\ag_n$ the bound $(\pi_\ell)_*\mu(B(x,r))\leq Cr^t$ holds, as will be shown at the end of the section.
To that end, we start by estimating the random variables $Z(S, n+1)$ conditional on $\mathcal F_n$ and bound the probability of the events $A_{n+1}^c$.
Using similar ideas, we then provide estimates for the probability of $G_{n+1}^c$ conditional on $\mathcal{F}_n$.

\begin{lemma}\label{lem:inside level}
For any strip $S$ with $r_{n+1}<w(S)\le r_n$, we have
\begin{align*}
\PP(Z&(S,n+1)> 500 R P_{n+1}w(S)^t\mid \mathcal F_n)\\
&\leq \exp(-500 R w(S)^tP_{n+1}+20 w(S) r_n^{-1}N_{n+1}Z(S,n)).
\end{align*}
\end{lemma}

\begin{proof}
Denote $w(S)$ by $\delta$ and $Z(S,n)$ by $K$. Let $\{ Q_{1}, Q_{2},\dots ,Q_{K}\}$ be the $r_{n}^{-1}$-adic squares from $\mathcal{F}_{n}$ that hit S, and for $i=1,\dots, K$ let $\{ Q_{i,1}, Q_{i,2}\cdots Q_{i, m_{i}}\}$ be the $r_{n+1}^{-1}$-adic squares inside $Q_{i}$ that hit $S$. Conditional on $\mathcal F_n$ these notions are deterministic. For each  $1\leq i\leq K$, let $(U_1^i,\dots,U_{N_{n+1}}^i)=\{U\in\mathcal F_{n+1}\mid U\subset Q_i\}$. Further, let $X_{i,j}, 1\leq j \leq N_{n+1}$ be the random variables defined as
\[
X_{i,j}=
\begin{cases}
1, \textrm{ when }U_j^i\cap S\neq \emptyset\,,\\
0, \textrm{ otherwise\,.}
\end{cases}
\]
Then, conditional on $\mathcal F_n$, $X_{i,j}$ are independent random variables and take value $1$ with probability  $m_{i}M_{n+1}^{-2}$, where $m_i$ is the number of $r_{n+1}^{-1}$-adic subsquares of $Q_i$ that touch $S$.  Further, $$Z(S, n+1) = \sum_{i=1}^{K}\sum_{j=1}^{N_{n+1}} X_{i,j}.$$
It is easy to see that $m_{i}M_{n+1}^{-2}\leq 5\delta \sqrt{2}r_{n}^{-1}$, so that  $\PP( X_{i,j}=1\mid \mathcal F_n) \leq 5\delta\sqrt{2}r_{n}^{-1}$.

Applying Markov's inequality to the random variable $e^{ Z \left(S, n+1 \right)}$ gives
\begin{equation}\label{eq:needed}
\begin{aligned}
\PP\left(Z(S, n+1) > 500R\delta^{t}P_{n+1}\mid\mathcal F_n
 \right)
%&=\mathbb{P} \left(e^{ Z \left(S, n+1 \right)}
%\ge e^{ 500 R \delta^{t}
%P_{n+1}}\mid \mathcal F_n
%\right) \\
&\leq e^{-500 R \delta^{t}P_{n+1}
} \mathbb{E} \left(e^{ Z\left(S, n+1\right)} \mid\mathcal F_n\right).
\end{aligned}
\end{equation}
Since $X_{i,j}$ are mutually independent so are the $e^{ X_{i,j}}$, and
\[
\mathbb{E}( \prod_{i,j} e^{X_{i,j}}\mid\mathcal F_n ) = \prod_{i,j}\mathbb{E}(e^{ X_{i,j}}\mid\mathcal F_n)= \prod_{i,j} \left( e m_{i}M_{n+1}^{-2} +\left(1-m_{i}M_{n+1}^{-2} \right) \right).
\]
Using the fact that $( 1+ x) \leq e^{x}$ for all $x$, and $m_{i}M_{n+1}^{-2} \leq  5\delta\sqrt{2}r_{n}^{-1}$, we have
\[
\mathbb{E} \left(e^{ Z\left(S, n+1 \right)} \mid \mathcal F_n\right)
\leq  \exp \left( \left(e-1 \right)5\delta\sqrt{2}r_{n}^{-1}N_{n+1}
K\right).
\]
Combining the above calculations with \eqref{eq:needed} and recalling $\delta=w(S)$, $K=Z\left(S,n\right)$ gives
\begin{align*}
&\PP (Z(S, n+1) > 500 R w(S)^{t}P_{n+1}\mid\mathcal F_n)\\
&\leq \exp\left(- 500 R w(S)^{t} P_{n+1}\right) \exp\left(\left(e-1 \right)5\sqrt{2}w(S) r_{n}^{-1}N_{n+1}Z(S,n) \right)
\end{align*}
finishing the proof.
\end{proof}

\begin{notation}\label{collection}
Consider the collection given by Lemma \ref{lem:finite_approx_family} for $M=5 r_{n+1}^ {-1}$. Denote by $\mathcal S_n$ the strips in this collection with width bounded from above by $r_n$. Notice that by the construction in Lemma \ref{lem:finite_approx_family} all strips in collection $\mathcal S_n$ have width bounded from below by $r_{n+1}$.
\end{notation}
\begin{proposition}\label{prop:strips}
There is a constant $C_3>0$ such that for all large $n$
\[
\EE(\chi_{\cap_{k\leq n}\ag_k}\PP(A_{n+1}^c\mid \mathcal F_{n}))\le \PP(\bigcap_{k=1}^n\ag_k)2000 r_{n+1}^{-3}\exp(-C_3r_{n+1}^{\epsilon}).
\]
\end{proposition}
\begin{proof}
Condition on $\cap_{k\leq n}\ag_k$.
Notice that for large $n$, when $\ag_n$ and thus $G_n$ occurs, then by Lemma \ref{lem:lengthtonum} and \eqref{c_n},
\begin{equation}\label{upperboundZSn}
Z(S,n)\le 4(1+2\sqrt 2) R P_nr_n^t\leq 20R P_n r_n^t
\end{equation}
for all strips $S$ of width $r_{n+1}<w(S)\le r_n$. Recall that $0<\epsilon<s-t$ and hence for all large enough $n$ we have $P_{n+1}\geq r_{n+1}^{-t-\epsilon}$. Let $n$ be at least this large. When $G_n$ occurs, for a strip $S$ with $r_{n+1}<w(S)\le r_n$, we obtain using Lemma \ref{lem:inside level} and \eqref{upperboundZSn}
\begin{align*}
\PP(Z&(S,n+1)> 500 RP_{n+1}w(S)^t\mid \mathcal F_n)\\
&\leq \exp(-500 R w(S)^tP_{n+1}+20 w(S) r_n^{-1}N_{n+1}Z(S,n))\\
&\leq \exp(-500 R w(S)^tP_{n+1}+400 w(S) r_n^{-1}N_{n+1} R P_nr_n^t).
\end{align*}
Plugging in $P_{n+1}=N_{n+1}P_n$, this is bounded from above by
%$$\PP(Z(S,n+1)\leq
$$\exp(-w(S)^tP_{n+1}R(500-400 w(S)^{1-t}r_n^{t-1}))
\leq \exp(-C_3r_{n+1}^{-\epsilon})$$
for a positive constant $C_3$. Fix an $\omega$ such that $Z(\widetilde S,n+1)\le 500 RP_{n+1}w(\widetilde S)^t$ holds for all $\widetilde S\in\mathcal S_{n}$ where $\mathcal S_{n}$ is given in Notation \ref{collection}.
Then, for any strip $S$ with $\tfrac 15 r_{n+1}<w(S)\le \tfrac 15 r_n$, by Lemma \ref{lem:finite_approx_family}, there is a strip $\widetilde S\in\mathcal S_{n}$ with $S\subset \widetilde S$ and $w(\widetilde S)\leq 5w(S)$. Thus
\[
Z(S,n+1)\leq Z(\widetilde S,n+1)\leq 500\cdot 5^ tRP_{n+1}w(S)^t,
\]
and further, $A_{n+1}$ occurs. Apply Lemma \ref{lem:inside level} and Lemma \ref{lem:finite_approx_family} to all strips $\widetilde S$ in the collection $\mathcal S_n$, to obtain from the above calculation that
\[
\EE(\chi_{\cap_{k\le n}\ag_k}\PP(A_{n+1}^c\mid \mathcal F_{n}))\le \PP(\bigcap_{k=1}^n\ag_k)2000 r_{n+1}^{-3}\exp(-C_3r_{n+1}^{-\epsilon}).
\]
\end{proof}

Our next task is to adjust the argument of Lemma \ref{lem:inside level} in order to get, with high probability, a good upper bound for $|\ell\cap\mathcal{F}_{n+1}|$.

\begin{lemma}\label{lemma:probab_estimate}
For all $0<\lambda<(r_{n+1}\sqrt 2)^{-1}$, we have the bound
\begin{align*}
\PP(&|\ell\cap \mathcal F_{n+1}|> P_{n+1}Rr_{n+1}^{t+1}\mid \mathcal F_n)\\
&\leq \exp(-\lambda P_{n+1}Rr_{n+1}^{t+1})\exp\left( \frac{2\lambda N_{n+1}|\ell\cap \mathcal F_n|}{M_{n+1}^2(2- r_{n+1}\lambda\sqrt 2)} \right).
\end{align*}
\end{lemma}
\begin{proof}
Condition on $\mathcal{F}_n$. Let $Q_{1}, Q_{2}, \cdots , Q_{K}$ ($K=K(\ell,\mathcal F_n)$) be the squares in $\mathcal F_{n}$ hitting $\ell$ and $L_{i}=\mathcal H^1(\ell\cap Q_{i})$. For each $Q_{i}$, let $Q_{i,1},Q_{i,2}, \cdots, Q_{i, m_{i}}$ denote all the $r_{n+1}^{-1}$-adic subsquares of $Q_{i}$ touching $\ell$ and put $L_{i,j} = \mathcal H^1(\ell\cap Q_{i,j})$.
We have
\[
|\ell\cap \mathcal F_n| = \sum_{i=1}^{K} L_{i}\textrm{ and }
L_{i}=\sum _{j=1}^{m_{i}} L_{i,j}.
\]

For each  $1\leq i\leq K$, let $(U_1^i,\dots,U_{N_{n+1}}^i)=\{U\in\mathcal F_{n+1}\mid U\subset Q_i\}$. Then let $X_{i,k}, 1\leq k \leq N_{n+1}$ be random variables with
\[
X_{i,k}=\mathcal H^1(\ell\cap U_k^i).
\]
Notice that, conditional on $\mathcal F_n$, $X_{i,k}$ are independent. We may write
\[
|\ell\cap \mathcal F_{n+1}|= \sum_{i=1}^{K} \sum_{k=1}^{N_{n+1}}X_{i,k}.
\]

Let $0<\lambda<(r_{n+1}\sqrt 2)^{-1}$. We apply Markov's inequality for the random variable $e^{\lambda |\ell\cap \mathcal F_{n+1}|}$ to obtain an estimate
\begin{equation}\label{eq:markov}
\begin{aligned}
 \PP (|\ell\cap \mathcal F_{n+1}|>P_{n+1}R r_{n+1}^{t+1}&\mid \mathcal F_n)\\
 &\leq e^{-\lambda P_{n+1}R r_{n+1}^{t+1}} \mathbb{E} \left(  e^{\lambda |\ell\cap \mathcal F_{n+1}|} \mid \mathcal F_n\right)
\end{aligned}
\end{equation}

Now we estimate $ \mathbb{E} \left( e^{\lambda |\ell\cap \mathcal F_{n+1}| }\mid \mathcal F_n\right)$. Firstly, notice that
\[
\mathbb{E} \left( e ^{\lambda X_{i,k}} \mid \mathcal F_n\right)= 1- \frac{m_{i}}{M_{n+1}^{2}} + \frac{1}{M_{n+1}^{2}}\sum_{j=1}^{m_{i}} e^{\lambda L_{i,j}}.
\]
For all $ \vert x \vert < 2$, we use the fact
$e^{x} \leq  1 +2x/(2-x)$ and
%\]
$\lambda L_{i,j}< 2$, to obtain
\[
e^{\lambda L_{i,j} } \leq 1+\lambda L_{i,j} \dfrac{2}{2-\lambda  r_{n+1}\sqrt{2}},
\]
and further
\begin{equation}\label{eq:phase}
\begin{aligned}
\mathbb{E} \left( e ^{\lambda X_{i,k}} \mid \mathcal F_n\right)&\leq 1 + \lambda L_{i} \frac{1}{M_{n+1}^{2}} \cdot\dfrac{2}{2-\lambda  r_{n+1}\sqrt{2}}\\
 &\leq \exp \left( \frac{\lambda L_{i}}{M_{n+1}^{2}} \cdot\dfrac{2}{2-\lambda r_{n+1}\sqrt{2}}  \right).
\end{aligned}
\end{equation}

Since $X_{i,k}$ are independent, \eqref{eq:phase} yields
\begin{align*}
\mathbb{E} \left( e^{|\ell\cap \mathcal F_{n+1}| }\mid \mathcal F_n\right)&= \prod _{i,k}\mathbb{E} \left( e ^{\lambda X_{i,k}}\mid \mathcal F_n\right)\\
&\leq  \exp \left( \frac{\lambda N_{n+1}|\ell\cap \mathcal F_{n}|  }{M_{n+1}^{2}} \cdot\frac{2}{2-\lambda r_{n+1}\sqrt{2}}\right).
\end{align*}
Combining this with \eqref{eq:markov} finishes the proof.
\end{proof}

Next we estimate the probability of the event $G_{n+1}^c$.

\begin{proposition}\label{prop:lines}
There is a constant $C_2>0$ independent of $\epsilon$, and $N=N(\epsilon)\in\mathbb N$ such that for all $n\geq N$,
\[
\EE(\chi_{\cap_{k\le n}\ag_k}\PP(G_{n+1}^c\mid \mathcal F_{n}))\le \PP(\bigcap_{k=1}^n\ag_k) C  r_{n+1}^{-4}\exp(-C_2r_{n+1}^{-\epsilon})
\]
where $C$ is from Lemma \ref{lem:finitenumblines}.
\end{proposition}
\begin{proof}
To begin with, recall that when $\ag_{n}$ occurs, so does $G_n$ and then for all lines $\ell$
\begin{equation}\label{upperboundFn}
|\ell\cap \mathcal F_n|\leq P_nRr_n^ {t+1} + r_n.
\end{equation}
Further, when $G_n$ occurs, by Lemma \ref{lemma:probab_estimate} and \eqref{upperboundFn}, the probability $\PP(|\ell\cap \mathcal F_{n+1}|> P_{n+1}Rr_{n+1}^{t+1}\mid \mathcal F_n)$ is bounded from above by
\begin{align*}
&\exp(-\lambda P_{n+1}Rr_{n+1}^{t+1})\exp\left( \frac{2\lambda N_{n+1}|\ell\cap \mathcal F_n|}{M_{n+1}^2(2- r_{n+1}\lambda\sqrt 2)} \right)\\
&\leq \exp(-\lambda P_{n+1}Rr_{n+1}^{t+1})\exp\left( \frac{2\lambda N_{n+1}r_nP_{n}Rr_n^t(1 + P_n^{-1}R^{-1}r_n^{-t})}{M_{n+1}^2(2- r_{n+1}\lambda\sqrt 2)} \right)\\
&\leq \exp\left( -\lambda P_{n+1} Rr^{t+1}_{n+1}\left(1-
\frac{2 M_{n+1}^{t-1}(1 +P_n^{-1}R^{-1}r_n^{-t} )}{(2-r_{n+1}\lambda\sqrt 2)}\right)\right ).
\end{align*}
Choose $\lambda=r_{n+1}^{-1+\epsilon}$. Recall that for large $n$, by \eqref{eq:def of s} and the choice of $\epsilon$ we have $P_{n}\geq r_{n}^{-t-2\epsilon}$. Then the term $1- \frac{2 M_{n+1}^{t-1}(1 +P_n^{-1}R^{-1}r_n^{-t} )}{(2-r_{n+1}\lambda\sqrt 2)}$ is bounded from below by a constant $C_2 > 0$ for large values of $n$.

Thus, for all large $n$ and for each line $\ell$, recalling $P_{n+1}\ge r_{n+1}^{-t-2\varepsilon}$, we arrive at the estimate
\begin{align*}
\PP(|\ell\cap \mathcal F_{n+1}|> P_{n+1}&R r_{n+1}^{t+1}\mid \mathcal F_n)\\
&\leq \exp\left( -\lambda r_{n+1}^{1-2\epsilon}C_2\right )= \exp({-C_2r_{n+1}^{-\epsilon}}).
\end{align*}

Fix an $\omega$ such that the estimate $|\tilde\ell\cap\mathcal F_{n+1}|\leq P_{n+1}Rr_{n+1}^{t+1}$ holds for all lines $\tilde\ell\in\mathcal A_{n+1}$, where $\mathcal A_{n+1}$ is given by Lemma \ref{lem:finitenumblines}. Then by Lemma \ref{lem:finitenumblines} for any line $\ell$ we find a line $\tilde \ell\in\mathcal A_{n+1}$ with
\[
|\ell\cap\mathcal F_{n+1}|\leq |\tilde\ell \cap\mathcal F_{n+1}| + r_{n+1}\le P_{n+1}R r_{n+1}^{t+1} + r_{n+1},
\]
that is, $G(\ell,n+1)$ occurs. Thus by above calculations
\begin{align*}
\EE(\chi_{\cap_{k\geq n}\ag_k}&\PP(G_{n+1}^c\mid \mathcal F_{n}))\\
&\leq\EE(\chi_{\cap_{k\ge n}\ag_k}\PP(|\ell\cap\mathcal F_{n+1}|\geq P_{n+1}R r_{n+1}^{t}\textrm { for some }\tilde\ell\in\mathcal A_{n+1}\mid \mathcal F_{n}))\\
&\le \PP(\bigcap_{k=1}^n\ag_k) C r_{n+1}^{-4}\exp(-C_2r_{n+1}^{-\epsilon})\,,
\end{align*}
as required.
\end{proof}

\begin{theorem}
We have $\PP(\bigcap_{k=1}^\infty\ag_k)>0$.
\end{theorem}
\begin{proof}
Choose $N$ so large that  $1-2000 r_{n}^{-3}\exp(-C_3r_n^{-\epsilon})-C r_{n}^{-4}\exp(-C_2r_n^{-\epsilon})$ are positive and the claims of Propositions \ref{prop:strips} and \ref{prop:lines} hold for $n\ge N$. Choosing the constant $R$ in Notation \ref{events} large enough we can make sure that $\PP(\bigcap_{i=1}^{N} \ag_i) >0$. Observe that the choice of the constants $C,C_2,C_3$ can be made independent of $R\ge 1$.
Then, by Propositions \ref{prop:strips} and \ref{prop:lines}, for all $m\geq N+1$,
\begin{align*}
\PP(\ag_m&\mid \bigcap_{k=1}^{m-1}\ag_k)=\frac{ \PP(\ag_m\cap\bigcap_{k=1}^{m-1}\ag_k)}{\PP( \bigcap_{k=1}^{m-1}\ag_k)}\\
&=\frac{1}{\PP( \bigcap_{k=1}^{m-1}\ag_k)}\EE(\chi_{\cap_{k\le m-1}\ag_k}\PP(\ag_m\mid \mathcal F_{m-1}))\\
&\geq \frac{1}{\PP( \bigcap_{k=1}^{m-1}\ag_k)}\EE(\chi_{\cap_{k\le m-1}\ag_k}(1-\PP(A_m^c\mid \mathcal F_{m-1})-\PP(G_m^c\mid \mathcal F_{m-1})))\\
&\ge 1-2000 r_{m}^{-3}\exp(-C_3r_m^{-\epsilon})-C r_{m}^{-4}\exp(-C_2r_m^{-\epsilon}).
\end{align*}
Iterating the observation
\[
\PP(\bigcap_{k=1}^n \ag_k)=\PP(\bigcap_{k=1}^{N}\ag_k)
\prod_{m=N+1}^n\PP(\ag_m\mid \bigcap_{i=1}^{m-1}\ag_i)\]
\[\ge \PP(\bigcap_{k=1}^{N}\ag_k)
\prod_{m=N+1}^n  \bigg(1-2000 r_{m}^{-3}\exp(-C_3r_m^{-\epsilon})- Cr_{m}^{-4}\exp(-C_2r_m^{-\epsilon}) \bigg)
\]
and noticing that the series
\[
\sum_{m=1}^\infty
2000 r_{m}^{-3}\exp(-C_3r_m^{-\epsilon})+C r_{m}^{-4}\exp(-C_2r_m^{-\epsilon})
\]
converges, finishes the proof.
\end{proof}

We finish the proof of Theorem \ref{main} by giving a proof to Lemma \ref{lem:localdim}.

\begin{localdim}
Let $\omega\in\bigcap _{k
=1}^\infty \ag_k$. Fix a line $\ell$, a point $x\in\ell$ and $r>0$. Let $n$ be number with $\tfrac 15 r_{n+1}< 2r\le \tfrac 15 r_n$ and $\ell'$ the line perpendicular to $\ell$ with $\ell\cap \ell'= \{x\}$, and
\[
S=\{y\in\mathbb R^2\mid d(y,\ell')< r\}.
\]
Then $S=\pi_\ell^{-1}(B(x,r))$. Since $\omega\in\ag_{n+1}$, the event $A_{n+1}$ occurs and, in particular, $A(S,n+1)$ occurs as well. Thus, taking \eqref{c_n} into account,
\[
\mu_{n+1}(S)\leq 500\cdot 5^t RP_{n+1}w(S)^tc_{n+1}r_{n+1}^2\leq 500  \cdot5^tRw(S)^t.
\]
For $k\ge n+1$, we have $\mu_k(Q)=\mu_{n+1}(Q)$ for all $Q\in\mathcal{F}_{n+1}$ and  thus
\[
\mu_k(\pi_\ell^{-1}(B(x,r)))\leq 500\cdot 5^t2^tRr^t,
\]
as well. Letting $k\to\infty$, the claim follows.\qed
\end{localdim}

\section{Application to random covering sets}\label{sect:application}

In this section Theorem \ref{main} is applied to the problem of calculating dimensions of projections of random covering sets. The random covering set is defined as a subset of torus, but then projected as a subset of plane. We first recall the definition of random covering sets with generating sequence of ball-like sets. A sequence $(g_n)$ is a sequence of ball-like subsets of the torus $\mathbb{T}^2$ if there exists a sequence of balls $B(x_n,\delta_n)\subset g_n$ such that $\limsup\limits_{n\to\infty}\frac{\rho_n}{\delta_n}<+\infty$,
 where $\rho_n$ is the diameter of $g_n$ and decreases to zero. More generally, we can replace this by the slightly weaker condition
\begin{equation}\label{ball-like}
 \lim_{n\to\infty}\frac{\log\rho_n}{\log\delta_n}=1\,.
\end{equation}
 Let $\xi_n$ be a sequence of independent random variables uniformly distributed on $\mathbb{T}^2$, and denote the induced probability space by $(\Omega,\mathbb P)$. The random covering set is defined as
 \[
 E:=\limsup_{n\to\infty}(\xi_n+g_n).
 \]
As shown in \cite[Proposition 4.7]{JJKLS}, the mass transference principle \cite[Theorem 2]{BeresnevichVelani06}, easily implies that almost surely,
  \[
 \dimH(E)=\inf\{s\ge 0: \sum_{n=1}^\infty \rho_n^s<\infty\}=\limsup_{n\to\infty}\frac{\log n}{-\log\rho_n}:=s_0\,,
 \]
 provided $s_0\le 2$.
 In the above formula, $\rho_n$ can be replaced by $\delta_n$
due to \eqref{ball-like}.

The following theorem is the main result of the paper.
\begin{theorem}\label{randomcoveringproj}
Let $s_0<1$. Then almost surely for every line $\ell$, we have
$$\dimH \pi_\ell(E)=s_0.$$
\end{theorem}
  \begin{proof}
 Since the upper bound $\dim_H\pi_\ell(E)\le s_0$ is clear, it suffices to show the almost sure lower bound $\dim_H\pi_\ell(E)\ge s_0$.

For simplicity, we assume that the balls $B(x_n,\delta_n)$ are centred at the origin. This will not cause loss of generality since they are to be translated by uniformly distributed random variables.  Denote by $B_n=\xi_n+B(0,\delta_n)=B(\xi_n,\delta_n)\subset\mathbb{T}^2$ the randomly translated balls
Then $\limsup\limits_{n\to\infty}B_n\subset E$.

Our aim is to construct, with positive probability, a random Cantor-like set $F\subset E$ in order to apply Theorem \ref{main}.

To that end, choose a sequence $\{n_k\}$ of natural numbers such that
\begin{align}\label{s0}
\lim_{k\to\infty}\frac{\log n_k}{-\log \delta_{n_k}}&=\limsup_{n\to\infty}\frac{\log n}{-\log \delta_n}=s_0\,,\\
\label{nkcondition}
n_k&\ge 256\,k^2n_{k-1}^2\delta_{n_{k-1}}^{-2}\,,\\
\label{nkcondition2}
\lim_{k\to\infty}&\frac{\log \delta_{n_{k-1}}}{\log \delta_{n_k}}=0.
\end{align}

We now construct the Cantor-like subset by induction. Our construction is a simplification of the one in \cite{JJKLS}, but we repeat the argument for the convenience of the reader. Let $N_1=\tilde{N}_1=\lfloor \frac{1}{2}n_1\rfloor$ and define $I(1,\mathbb{T}^2)=\{1,\ldots,N_1\}$.
Decompose $\mathbb{T}^2$ (which we identify with the unit square) into $M_1^2$ disjoint subsquares, where $M_1=\lceil 2\delta_{n_1}^{-1}\rceil$, and denote by $\mathcal{D}_1$ the collection of closed $M_1$-adic subsquares of $\mathbb{T}^2$. Here and hereafter, for all $x\in\mathbb R$, the notation $\lfloor x \rfloor$ is used for the largest integer smaller than $x$ and $\lceil x\rceil$ for the smallest integer larger than $x$. For each $i\in I(1,\mathbb{T}^2)$, pick $Q_i\in \mathcal{D}_1$ satisfying
\begin{equation}\label{choiceofcube}
\xi_i\in Q_i\subset B_i,
\end{equation}
Notice that, since $\delta_i$ are in decreasing order and the sidelength of the cubes satisfies $M_1^{-1}\le \frac{\delta_{n_1}}{2}$, such a choice is possible. In case there are more than one cubes satisfying \eqref{choiceofcube}, that is, $\xi_i$ is on the boundary of a cube, just choose any of them (Since the boundaries of the cubes are of zero Lebesgue measure, the event that any of the $\xi_i$ lies on the boundary of a cube has zero probability anyway). Denote the collection of such $Q_i$'s by $$\mathcal{C}_1=\{Q_i\in \mathcal{D}_1:\ i\in I(1,\mathbb{T}^2)\}.$$
For completeness (see the definition of $\Omega_k$ and $q_k$ for $k>1$ below), we define $\Omega_1=\Omega$ and $q_1=\mathbb{P}(\Omega_1)=1$.

We continue the random construction inductively as follows: Assume that numbers $M_1,\dots, M_{k-1}$ and $\tilde{N}_1,\dots, \tilde{N}_{k-1}$ have been defined and suppose that for all $i=1,\dots, k-1$ we are given a collection $\mathcal{C}_i$ of $\tilde{N}_i$ squares in $\mathcal D_i$, where $\mathcal{D}_i$ is the family of $(\prod_{j=1}^i M_j)$-adic subsquares of $\mathbb T^2$.

Denote
 $\mathcal{C}_{k-1}=\{\tilde{Q}_1,\cdots,\tilde{Q}_{\tilde{N}_{k-1}}\}$. Let $m_k=\lfloor \frac{n_k-n_{k-1}}{\tilde{N}_{k-1}} \rfloor$. For any $1\le l\le \tilde{N}_{k-1}$, consider the random family of indices consisting of those $n_{k-1}+(l-1)m_k+1\le j\le n_{k-1}+lm_k$, for which $\xi_j\in\tilde{Q}_l$. If there are at least $\lfloor \frac{1}{2}m_k\mathcal{L}(\tilde{Q}_l)\rfloor:=N_k$ such indices $j$, let $I(k,\tilde{Q}_l)$ denote the first $N_k$ of them. Let $\Omega_k$ be the event that $I(k,\tilde{Q}_l)$ is well defined for all $\tilde{Q}_l\in\mathcal{C}_{k-1}$.

Conditioning on $\Omega_k$, decompose every $Q\in\mathcal{D}_{k-1}$ to $M_k^2$ disjoint subsquares, where $M_k=\lceil 2\delta_{n_k}^{-1}(\Pi_{j=1}^{k-1}M_j)^{-1}\rceil$,  and denote by $\mathcal{D}_k$ the collection of closed $(\Pi_{j=1}^kM_j)$-adic subsquares of $\mathbb{T}^2$. For any $i\in I(k,\tilde{Q}_l)$, pick $Q_i\in \mathcal{D}_k$ satisfying
\begin{equation*}%\label{choiceofcubek}
\xi_i\in Q_i\subset B_i.
\end{equation*}
Since the sidelength of these squares is $(\Pi_{j=1}^kM_j)^{-1}\le \frac{\delta_{n_k}}{2}$, one can always find such $Q_i$, and it is unique outside a zero measure set of $\xi_i$:s.
 Set
$$\mathcal{C}_k=\{Q_i\in\mathcal{D}_k: Q\in\mathcal{C}_{k-1}, i\in I(k,Q)\}$$
%the collection of such $Q_i$'s
and $\tilde{N}_k:=\#\mathcal{C}_k=\tilde{N}_{k-1}N_k$.

For $\omega\in\Omega_\infty=\cap_{k\ge 1}\Omega_k$, consider the random Cantor set $F=F(\omega)$ defined as
$$F=\bigcap_{k=1}^\infty\bigcup_{Q\in\mathcal{C}_k}Q\subset\limsup_{n\to\infty}B_n\subset E\,.$$

We next check that $\mathbb{P}(\Omega_\infty)>0$. Using \cite[Proposition 2.6]{JJKLS} and the construction of $\mathcal{C}_k$, we have the following lower bound for the conditional probabilities $q_k:=\mathbb{P}(\Omega_k|\Omega_1,\dots,\Omega_{k-1})$,
$$q_k\geq 1-\tilde{N}_{k-1}\frac{4(1-\mathcal{L}(Q))}{m_k\mathcal{L}(Q)}\ge 1-\frac{4\tilde{N}_{k-1}(\Pi_{j=1}^{k-1}M_j)^2}{m_k}$$
where $Q\in \mathcal{C}_{k-1}$.
Note that since
\begin{equation}\label{mkestimate}
m_k\ge \frac{n_k-n_{k-1}}{2\tilde{N}_{k-1}}\ge \frac{n_k}{4\tilde{N}_{k-1}}
\end{equation}
 and $\Pi_{j=1}^{k-1}M_j\le 4\delta_{n_{k-1}}^{-1}$,
 we have
$$q_k\ge 1-\frac{256\tilde{N}_{k-1}^2\delta_{n_{k-1}}^{-2}}{n_k}\ge 1-\frac{256 n_{k-1}^2\delta_{n_{k-1}}^{-2}}{n_k}\ge 1-\frac{1}{k^2},$$
where the second inequality follows because $\tilde{N}_{k-1}\le n_{k-1}$ and the third one is due to \eqref{nkcondition}. Therefore, $\mathbb{P}(\Omega_\infty)=\Pi_{k=1}^\infty q_k>0.$

Now $F$ is a random Cantor set as defined in Section \ref{sect:preliminaries} with the defining sequences $(M_n)$ and $(N_n)$. Thus,
we can apply Theorem  \ref{main} to obtain with positive probability for all lines $\ell\subset\mathbb{R}^2$,
$$\dim_H\pi_\ell(E)\geq\dimH\pi_\ell(F)
=\liminf_{k\to\infty}\frac{\sum_{j=1}^k\log N_j}{\sum_{j=1}^k\log M_j}\,.
%=lta_{n_k\lim_{k\to\infty}\frac{\log n_k}{-\log \de}
$$
Since
$\prod_{j=1}^kM_j\le 4\delta_{n_k}^{-1}$ and using \eqref{mkestimate}
$$N_k\ge\frac{m_k}{4\prod_{j=1}^{k-1}M_{k}^2}\ge\frac{n_k\delta_{n_{k-1}}^2}{16 \tilde{N}_{k-1}}\ge\frac{n_k\delta_{n_{k-1}}^3}{16}\,,$$
recalling \eqref{nkcondition2} we have,
\[\frac{\sum_{j=1}^k\log N_j}{\sum_{j=1}^k\log M_j}\ge \lim_{k\to\infty}\frac{\log n_k}{-\log \delta_{n_k}}=s_0\,.\]
Finally, since
\[\dim_H\pi_\ell(E)\geq s_0\text{ for all lines }\ell\]
is a tail event of positive probability, the Kolmogorov zero-one law implies that it must have full probability.
\end{proof}
\begin{remark}
In \cite{JJKLS} and \cite{Persson}, dimension formulas are obtained for certain affine type random covering sets. For these classes of random covering sets
there is no similar result to Theorem \ref{randomcoveringproj}. Here is a counterexample:

 Let $(g_n)$ be a sequence of rectangles in $\mathbb{T}^2$ with the sidelength of the sides parallel to the $x$-axis $\alpha_1(g_n)=n^{-\alpha}$ and of those parallel to the $y$-axis $\alpha_2(g_n)=n^{-\beta}$, where $\beta>\alpha>1$. Then the Hausdorff dimension $\dim_H(E)=\alpha^{-1}$ almost surely. However, for the projection to the $y$-axis, $\pi_y$, instead $\dim_H(\pi_y(E))=\beta^{-1}$.
 \end{remark}

\section{Generalizations}\label{sect:general}
For simplicity, in above we have considered the case $d=2$ only.
In a similar way, one can define the random Cantor sets $F$ with defining sequences $(M_n)$, $(N_n)$, $N_n\le M_n^d$ and almost sure Hausdorff dimension
\[s=\liminf\limits _{n\to\infty} \frac{\sum_{i=1}^n\log N_i}{\sum_{i=1}^n\log M_i}\,\]
also on the higher dimensional torus $\mathbb{T}^d$ for any $d\ge3$. Embedding $\mathbb{T}^d$ in a natural way into the unit cube of $\mathbb{R}^d$ one can then consider orthogonal projections of $F$ into $k$-dimensional linear subspaces of $\mathbb{R}^d$. Let $\mathbb{G}_{d,k}$ denote the family of all $k$-dimensional linear subspaces of $\mathbb{R}^d$ and for $V\in\mathbb{G}_{d,k}$, let $\pi_V$ denote the orthogonal projection onto $V$. The proofs from Sections \ref{sect:preliminaries}--\ref{sect:probability} can be easily modified to this setting in order to obtain the following generalisation of Theorem \ref{main}.

\begin{theorem}\label{general}
If $s<k\le d$, then almost surely for every $V\in\mathbb{G}_{d,k}$, we have $\dimH\pi_V(F)=s$.
\end{theorem}

As in case $d=2$, this result can then be applied for the random covering sets of the torus in order to prove

\begin{theorem}
Let $E\subset\mathbb{T}^d$ be a ball-like random covering set with almost sure Hausdorff dimension $s_0< k$. Then almost surely for all $V\in\mathbb{G}_{d,k}$, $\dimH\pi_V(E)=s_0$.
\end{theorem}

\bibliography{randomcoveringCKLS}
\bibliographystyle{abbrv}
\end{document}